\documentclass[smallextended,envcountsect]{svjour3} 
\smartqed 
\usepackage{graphicx}
\journalname{JOTA}

\usepackage{amsmath}
\usepackage{amssymb}
\usepackage{theorem}
\usepackage{pstricks}
\usepackage{euscript}
\usepackage{epic,eepic}
\usepackage{graphicx}

\usepackage{amsmath}
\usepackage{amssymb}
\usepackage{theorem}

\newcommand{\menge}[2]{\big\{{#1} : {#2}\big\}}

\newcommand{\emp}{\ensuremath{{\varnothing}}}

\newcommand{\scal}[2]{\left\langle{#1}, {#2} \right\rangle}

\newcommand{\HH}{\ensuremath{\mathcal H}}

\newcommand{\RR}{\ensuremath{\mathbb R}}

\newcommand{\NN}{\ensuremath{\mathbb N}}

\newcommand{\dom}{\ensuremath{\operatorname{dom}}}

\newcommand{\gr}{\ensuremath{\operatorname{gra}}}

\newcommand{\ran}{\ensuremath{\operatorname{ran}}}
\newcommand{\zer}{\ensuremath{\operatorname{zer}}}
\newcommand{\gra}{\ensuremath{\operatorname{gra}}}

\newcommand{\Fix}{\ensuremath{\operatorname{Fix}}}
\newcommand{\Id}{\ensuremath{\operatorname{Id}}}

\newcommand{\weakly}{\ensuremath{\rightharpoonup}}

\begin{document}

\title{Finding the Forward-Douglas-Rachford-Forward Method
\thanks{Communicated by Jalal Fadili.}}
\author{Ernest K. Ryu \and B$\grave{\text{\u{a}}}$ng C\^ong V\~u}

\institute{Ernest K. Ryu, Corresponding author \at
             University of California, Los Angeles\\
              California, United States\\
		eryu@math.ucla.edu
           \and
           B$\grave{\text{\u{a}}}$ng C\^ong V\~u,    \at
           Department of Mathematics, Vietnam National University\\
              Hanoi, Vietnam \\
              bangcvvn@gmail.com
}

\date{
Received: date / Accepted: date}

\maketitle

\begin{abstract}
We consider the monotone inclusion problem with a sum of 3 operators, in which 2 are monotone and 1 is monotone-Lipschitz. The classical Douglas--Rachford and Forward-backward-forward methods respectively solve the monotone inclusion problem with a sum of 2 monotone operators and a sum of 1 monotone and 1 monotone-Lipschitz operators. We first present a method that naturally combines Douglas--Rachford and Forward-backward-forward and show that it solves the 3 operator problem under further assumptions, but fails in general. We then present a method that naturally combines Douglas--Rachford and forward-reflected-backward, a recently proposed alternative to Forward-backward-forward by Malitsky and Tam [arXiv:1808.04162, 2018]. We show that this second method solves the 3 operator problem generally, without further assumptions.
\end{abstract}



\section{Introduction}
\label{s:intro}
We consider the monotone inclusion problem of finding a zero of the sum of 2 maximal monotone and 1 monotone-Lipschitz operators. 
The classical Douglas--Rachford (DR) splitting by Lions and Mercier \cite{Lions1979} solves the problem with a sum of 2 maximal monotone operators.
The classical forward-backward-forward (FBF) splitting by Tseng \cite{Tseng00} solves the problem with a sum of 1 maximal monotone and 1 monotone-Lipschitz operators.
We consider the generalization of the setups of DR and FBF.

Recently, there has been much work developing splitting methods combining and unifying classical ones.
Another classical method is forward-backward (FB) splitting \cite{BRUCK1977159,Passty1979}, which solves the problem with a sum of 1 monotone and 1 cocoercive operators.
The effort of combining DR and FB was started by Raguet, Fadili, and Peyr\'e \cite{Raguet11,Raguet18}, extended by Brice\~no-Arias \cite{Luis15}, and completed by Davis and Yin \cite{Davis} as they proved convergence for the sum of 2 monotone and 1 cocoercive operators.
FB and FBF were combined by Brice\~no-Arias and Davis \cite{Luis16} as they proved convergence for 1 monotone, 1 cocoercive, and 1 monotone-Lipschitz operators.
These combined splitting methods can efficiently solve monotone inclusion problems with more complex structure.

On the other hand, DR and FBF have not been fully combined, to the best of our knowledge.
Banert's relaxed forward backward (in the thesis \cite{Banert12})
and 
Brice\~no-Arias's forward--partial inverse--forward \cite{Luis15JOTA}
 combine DR and FBF in the setup where one operator is a normal cone operator with respect to a closed subspace.
However, neither method applies to the general setup with 2 maximal monotone and 1 monotone-Lipschitz operators.

In this work, we first present a method that naturally combines and unifies DR and FBF. We prove convergence under further assumptions, and we prove, through a counterexample, that convergence cannot be established in full generality.
We then propose a second method that naturally combines and unifies DR and forward-reflected-backward (FRB), a recently proposed alternative to FBF by Malitsky and Tam \cite{Malitsky18}.
We show that this combination of DR and FRB does converge in full generality.

The paper is organized as follows.
Section~\ref{s:problem_statement} states the problem formally.
Section~\ref{s:prelim} reviews preliminary information and sets up the notation.
Section~\ref{s:algo} presents our first proposed method combining DR and FBF, proves convergence under certain further assumptions, and proves divergence in the fully general case.
Section~\ref{s:FRDR} presents our second proposed method combining DR and FRB and proves convergence in the fully general case.
Section~\ref{s:comparison} compares our presented method with other similar and relevant methods.

\section{Problem Statement, Contribution, and Prior Work}
\label{s:problem_statement}
Consider the monotone inclusion problem
\begin{equation}
\mbox{find $x\in \HH$ \quad such that \quad }0\in Ax+Bx+Cx,
\label{e:prob1}
\end{equation}
where $\HH$ is a real Hilbert space.
Throughout, we assume for some $\mu\in]0,\infty[$:
\makeatletter
\tagsleft@true
\makeatother
\begin{gather}
\text{$A\colon\HH\rightrightarrows \HH$ and $B\colon\HH\rightrightarrows \HH$ are maximal monotone.}
\tag{A1}\label{assump:a1}\\
\text{$C\colon\HH\rightarrow\HH$ is monotone and $\mu$-Lipschitz continuous.}
\tag{A2}\label{assump:a2}\\
\text{$\zer(A+B+C)$ is not empty.}
\tag{A3}\label{assump:a3}
\end{gather}
\makeatletter
\tagsleft@false
\makeatother

Let $J_{\gamma A}$, $J_{\gamma B}$, and $J_{\gamma C}$, respectively, denote the resolvents with respect to $A$, $B$, and $C$ with parameters $\gamma$.
 We informally assume $J_{\gamma A}(x)$, $J_{\gamma B}(x)$, and $C(x)$  can be evaluated efficiently for any input $x\in \HH$.
 However, $J_{\gamma C}(x)$ may be difficult to evaluate.
Therefore, we restrict our attention to methods that  activate $C$ through direct evaluations, rather than through $J_{\gamma C}$. 
The monotone-Lipschitz operator $C$ of Problem~\eqref{e:prob1} arises 
as skew linear operators primal-dual optimization \cite{plc6,Combettes13}
and saddle point problems \cite{Rockafellar1970_saddle}.


When $C=0$, we can use the classical Douglas--Rachford (DR) splitting presented by Lions and Mercier \cite{Lions1979}:
\begin{align*}
\begin{cases}
x_{n+1}=J_{\gamma B}(z_{n})\\
y_{n+1}=J_{\gamma A}(2x_{n+1}-z_{n})\\
z_{n+1}=z_{n}+y_{n+1}-x_{n+1},
\end{cases}
\end{align*}
where the step size parameter satisfies $\gamma \in \left]0,+\infty\right[$.
When $B=0$, we can use the classical forward-backward-forward (FBF) splitting by Tseng \cite{Tseng00}:
\begin{align*}
\begin{cases}
y_{n+1}=J_{\gamma A}(x_{n}-\gamma Cx_{n})\\
x_{n+1}=y_{n+1} -\gamma (Cy_{n+1}-Cx_{n}),
\end{cases}
\end{align*}
 where the step size parameter satisfies $\gamma \in \left]0,1/\mu \right[$.
Recently, Malitsky and Tam \cite{Malitsky18} have proposed
forward-reflected-backward (FRB) splitting, another method for the case $B=0$:
\begin{align*}
\begin{cases}
x_{n+1}=J_{\gamma A}(x_{n}-\gamma (2Cx_{n}-Cx_{n-1})),
\end{cases}
\end{align*}
 where the step size parameter satisfies $\gamma \in\left]0,1/(2\mu)\right[$.

The contribution of this work is the study of splitting methods combining DR with other methods to incorporate an additional monotone-Lipschitz operator.
We characterize to what extent DR+FBF works and to what extent it fails. We then demonstrate that DR+FRB is a successful combination.



Several other 3-operator splitting methods
have been presented in recent years.
Combettes and Pesquet's PPXA \cite{plc6},
Bo{\c t}--Hendrich \cite{Bot2013},
Latafat and Patrinos's AFBA \cite{Latafat2018}, and 
Ryu's 3-operator resolvent-splitting \cite{Ryu2018}
solve the problem with 3 or more monotone operators by activating the operators through their individual resolvents.
Condat--V\~u \cite{Condat2013,Vu2013}, FDR \cite{Raguet11,Raguet18,Luis15,Davis}, and Yan's PD3O \cite{Yan2018}
solve the problem with 2 monotone and 1 cocoercive operators by activating the 2 monotone operators through their resolvents and the cocoercive operator through forward evaluations.
FBHF \cite{Luis16} solves the problem with 1 monotone, 1 cocoercive, and 1 monotone-Lipschitz operators by activating the monotone operator through its resolvent and the cocoercive and monotone-Lipschitz operators through forward evaluations.
These methods do not apply to our setup since we have 2 monotone operators, which we activate through their resolvents, and 1 monotone-Lipschitz operator, which we activate through forward evaluations.

The primal-dual method by Combettes and Pesquet \cite{plc6} and the instances of projective splitting by Johnstone and Eckstein \cite{johnstone2018,johnstone2018b} are existing methods that do solve Problem~\eqref{e:prob1}.
However, these methods do not reduce to DR. We compare the form of these methods in Section~\ref{s:comparison}.

While this paper was under review, there have been exciting developments on splitting methods based on cutting planes (separating hyperplanes): Warped proximal iterations by B\`ui and Combettes \cite{combettes2019} and NOFOB by Giselsson \cite{giselsson2019} are general frameworks that can solve Problem~\eqref{e:prob1}. The methods that arise from these frameworks are different from the methods we present.




\section{Preliminaries}
\label{s:prelim}
In this section, we quickly review known results and set up the notation.
The notation and results we discuss are standard, and interested readers can find further information in \cite{livre1,primer}.

Write $\HH$ for a real Hilbert space and, respectively, write $\scal{\cdot}{\cdot}$ and $\|\cdot\|$ for its associated scalar product and norm.
Write $\Id\colon\HH\to\HH$ for the identity operator.
Write $A\colon\HH\rightrightarrows \HH$ to denote that $A$ is a set-valued operator.
For simplicity, we also write $Ax := A(x)$.
When $A$ maps a point to a singleton, we also write $Ax = y$ instead of $Ax = \{y\}$.
Write $\dom(A):=\menge{x\in\HH}{Ax\not=\emp}$ for the domain of $A$
and $\ran(A) := \menge{u\in\HH}{(\exists x\in\HH)\, u\in Ax }$ for the range of $A$.
Write $\gra(A) := \menge{(x,u)\in\HH\times\HH}{u\in Ax}$ for the graph of $A$.
The inverse of $A$ is the set-valued operator defined by $A^{-1}\colon u \mapsto \menge{x}{ u\in Ax}$. 
The zero set of $A$ is $\zer(A) := A^{-1}0$.
We say that $A$ is  monotone if 
\begin{equation*}\label{oioi}
\big(\forall (x,u), (y,v)\in\gra A\big)
\quad\scal{x-y}{u-v}\geq 0,
\end{equation*}
and it is maximally monotone if there exists no monotone operator $B$ such that $\gra(B)$ properly contains $\gra(A)$.
The resolvent of $A$ is
$ J_A:=(\Id + A)^{-1}$.
When $A$ is maximal monotone, $J_A$ is single-valued and $\dom J_A=\HH$.
A single-valued operator $B\colon\HH\to\HH$ is $\kappa$-cocoercive for $\kappa\in ]0,\infty[$ if
\begin{equation*}
(\forall x,y\in\HH)\quad \scal{x-y}{Bx-By} \geq \kappa\|Bx-By\|^2.
\end{equation*}  
A single-valued operator $C\colon\HH\rightarrow \HH$ is $\mu$-Lipschitz for $\mu\in]0,\infty[$ if
\[
\left(\forall x,y\in \HH\right)
\quad
\|Cx-Cy\|\le\mu \|x-y\|.
\]
A single-valued operator $R\colon\HH\rightarrow \HH$ is nonexpansive if
\[
\left(\forall x,y\in \HH\right)
\quad
\|Rx-Ry\|\le \|x-y\|,
\]
i.e., if $R$ is $1$-Lipschitz. Let $\theta\in \left[0,1\right]$.
A single-valued operator $T\colon\HH\rightarrow \HH$ is $\theta$-averaged if
$T=(1-\theta)I+\theta R$ for some nonexpansive operator $R$.
Define the normal cone operator with respect to a nonempty closed convex set $C\subseteq \HH$ as
\[
N_C(x) = \left\{ \begin{array}{ll}
\emp, & \text{if}\; x \not\in C\\
\menge{ y\in\HH}{\scal{y}{z-x}\leq 0~\forall z\in C}, &  \text{if}\;  x \in C.
\end{array}\right.
\]

The Cauchy--Schwartz inequality states
$\scal{u}{v}\le \|u\|\|v\|$
for any $u,v\in \HH$.
The Young's inequality states
\[
\scal{u}{v}\le \frac{\eta}{2}\|u\|^2+\frac{1}{2\eta}\|v\|^2,
\]
for any $u,v\in \HH$ and $\eta>0$.

\begin{lemma} \label{l:21}
If $C\colon\HH\to\HH$ is $\mu$-Lipschitz continuous, 
then $\Id-\gamma C$ is one-to-one for $\gamma\in]0,1/\mu[$.
\end{lemma}
\begin{proof}
Although this result follows immediately from the machinery of scaled relative graphs \cite{Ryu2019}, we provide a proof based on first principles.
Let $x,y\in\HH$. Then $ \|Cx-Cy\| \leq \mu \|x-y\|$ and
\begin{align*}
\|(\Id-\gamma C)x-(\Id-\gamma C)y\|
\ge  \|x-y\| -\gamma\|Cx-Cy\|
\ge(1-\gamma \mu)\|x-y\|,
\end{align*}
by Cauchy--Schwartz and $\mu$-Lipschitz continuity.
Thus $(\Id-\gamma C)x=(\Id-\gamma C)y$ if and only if $x=y$.
\qed\end{proof}

A classical result states that $J_{B}$ is $(1/2)$-averaged if $B$ is maximal monotone \cite[Proposition 23.8]{livre1}.
The following lemma states that $J_{B}$ is furthermore $\theta$-averaged with $\theta<1/2$ if $B$ is cocoercive.
\begin{lemma} 
\label{l:2}
Let $\gamma,\kappa\in]0,\infty[$.
If $B:\HH\rightarrow\HH$ is $\kappa$-cocoercive, then $J_{\gamma B}$ is $\frac{1}{2(1+\kappa/\gamma)}$-averaged
and 
\[
\|J_{\gamma B}x-J_{\gamma B}y\|^2\le
\|x-y \|^2   - \left(1 + \frac{2\kappa}{\gamma}\right) \|(\Id-J_{\gamma B})x- (\Id-J_{\gamma B})y \|^2.
\]
for any $x,y\in\HH$.
\end{lemma}
\begin{proof}
Although this result follows immediately from the machinery of scaled relative graphs \cite{Ryu2019}, we provide a proof based on first principles.
Let $u=J_{\gamma B}x$ and $v=J_{\gamma B}y$, i.e., $\gamma^{-1}(x-u)=  Bu$ and $\gamma^{-1}(y-v)= Bv$.
Since $B$ is $\kappa$-cocoercive, we have
\[
\kappa \|\gamma^{-1}(x-u) -\gamma^{-1}(y-v)  \|^2 \leq \scal{u-v}{\gamma^{-1}(x-u) -\gamma^{-1}(y-v)}.
\]
This implies
\begin{align*}
(\kappa/\gamma) \|x-u -y+v  \|^2 &\leq  \scal{u-v}{x-u -y+v}\notag\\
&= \scal{u-v}{x-y} - \|u-v\|^2\notag\\
&= -\frac{1}{2} \|u-v\|^2 +\frac{1}{2} \|x-y\|^2 - \frac{1}{2}\|x-u -y+v \|^2,
\end{align*}
which proves the stated inequality.
Finally, this inequality is equivalent to $\frac{1}{2(1+\kappa/\gamma)}$-averagedness of $J_{\gamma B}$ by \cite[Proposition 4.35]{livre1}.
\qed\end{proof}

\section{FBF+DR: Convergent with Further Assumptions}
\label{s:algo}
To solve Problem~\eqref{e:prob1}, we propose the following iteration
\makeatletter
\tagsleft@true
\makeatother
\begin{align}
\qquad\qquad\qquad
\begin{cases}
x_{n+1}= J_{\gamma B}z_n\\
y_{n+1}= J_{\gamma A}(2x_{n+1}- z_n -\gamma Cx_{n+1})\\
z_{n+1}= z_n+y_{n+1}-x_{n+1}-\gamma (Cy_{n+1}-Cx_{n+1})
\end{cases}
\tag{FDRF}
\label{e:FDRF}
\end{align}
\makeatletter
\tagsleft@false
\makeatother
for $n=0,1,\dots$
where $z_0\in\HH$ is a starting point and $\gamma>0$.
We call this method forward-Douglas--Rachford-forward (FDRF) splitting as it combines Tseng's FBF \cite{Tseng00} and Douglas--Rachford \cite{Lions1979};
FDRF reduces to FBF when $B=0$ and to DR when $C=0$.



We can view FDRF as a fixed-point iteration $z_{n+1}=Tz_n$ with
\[
T:= (\Id-\gamma C) J_{\gamma A}(2J_{\gamma B} -\Id- \gamma C J_{\gamma B}) +\Id - (\Id-\gamma C) J_{\gamma B}.
\]
The following result states that $T$ is a fixed-point encoding for Problem~\eqref{e:prob1}.
\begin{lemma}\label{lem:fixed-piont}
Assume \eqref{assump:a1} and \eqref{assump:a2}.
If $\gamma\in]0,1/\mu[$, then
\[
\zer(A+B+C) = J_{\gamma B}(\Fix(T)),
\]
where $\Fix(T):=\menge{x\in\HH}{Tx=x}$.
\end{lemma}
\begin{proof}
Let $x\in\zer(A+B+C)$. Then, there exists $u\in Ax$ and $v\in Bx$ such that $0= u+v+Cx$.
It follows from $v\in Bx$ that $x = J_{\gamma B}z$ where $z= x+\gamma v$. 
We have $2 J_{\gamma B}z-z-\gamma CJ_{\gamma B}z= 2x-z-\gamma Cx = x+\gamma u\in (\Id+\gamma A)x$ and  
$x= J_{\gamma A}(2 J_{\gamma B}z-z-\gamma CJ_{\gamma B}z)$. Therefore,
\[
(\Id-\gamma C)J_{\gamma B}z = (\Id-\gamma C) J_{\gamma A}(2 J_{\gamma B}z-z-\gamma CJ_{\gamma B}z),
\]
which shows that $Tz=z$ and $\zer(A+B+C) \subset J_{\gamma B}(\Fix(T))$. Now, let $z\in \Fix(T)$. By Lemma \ref{l:21},
we have $J_{\gamma B}z=J_{\gamma A}(2 J_{\gamma B}z-z-\gamma CJ_{\gamma B}z)$.
Set $x= J_{\gamma B}z$. Then, $z-x\in \gamma Bx$ and $(2x-z-\gamma Cx)-x \in \gamma Ax$. Therefore, 
$0\in Ax +Bx+Cx$ and hence $J_{\gamma B}(\Fix(T))\subset \zer(A+B+C)$.
\qed\end{proof}



Under further assumptions, FDRF's $(x_n)_{n\in\NN}$ sequence converges weakly to a solution of \eqref{e:prob1}.
\begin{theorem} \label{t:1}
Assume \eqref{assump:a1}, \eqref{assump:a2}, and \eqref{assump:a3}.
If furthermore one of the following conditions holds
\begin{itemize}
    \item[(i)]  $B$ is $\kappa$-cocoercive and $\gamma\in \big]0,\mu^{-1}/\sqrt{1 +\gamma/(2\kappa))}\big[$, which is satisfied, for example, if $0 < \gamma < \min\{\kappa, \mu^{-1}\sqrt{2/3}\}$. 
    \item[(ii)] $B= N_{V}$ and $C = P_VC_1P_V$ for some  closed vector space $V$ and single-valued operator $C_1\colon\HH\to\HH$ and $\gamma\in]0,1/\mu[$,
\end{itemize}
then $z_n\weakly z_\star \in\Fix(T)$ and $x_n\weakly J_{\gamma B}z_\star\in \zer(A+B+C)$ and $y_n\weakly J_{\gamma B}z_\star$ for \eqref{e:FDRF}. 
\end{theorem}

\begin{proof}
Let $z_\star\in\Fix(T)$ and $x_\star :=J_{\gamma B}z_\star$.
Set $u_{n+1} := x_{n+1} - z_n +z_\star-x_\star$ and we have
\begin{alignat*}{2}
z_{n+1} - z_\star 
&= z_n + y_{n+1} - x_{n+1}- \gamma (Cy_{n+1}-Cx_{n+1}) - z_\star\notag\\
&= y_{n+1}-x_\star + \gamma (Cx_{n+1}-Cy_{n+1})-u_{n+1}.
\end{alignat*} 
By 
\begin{align}
\|z_{n+1}-z_\star \|^2 &= \|y_{n+1} - x_\star +\gamma (Cx_{n+1}-Cy_{n+1})\|^2
\nonumber
\\
&\quad-2\scal{y_{n+1} - x_\star +\gamma (Cx_{n+1}-Cy_{n+1})}{u_{n+1}} + \|u_{n+1}\|^2.
\label{e:eq0}
\end{align}
We expand the first term to get
\begin{align}
\label{e:eq1}
&\|y_{n+1} - x_\star +\gamma (Cx_{n+1}-Cy_{n+1})\|^2 \\
&= \| y_{n+1}-x_\star\|^2 + 2\gamma \scal{y_{n+1}-x_\star}{Cx_{n+1}-Cy_{n+1}} + \gamma^2\|Cx_{n+1}-Cy_{n+1} \|^2.\nonumber
\end{align}
Note 
\begin{align*}
2x_{n+1} -z_n - \gamma Cx_{n+1} -y_{n+1}
\in&\gamma Ay_{n+1},\\
x_\star -z_\star- \gamma Cx_\star 
\in&\gamma Ax_\star.
\end{align*}
Since $A$ and $C$  are monotone, we have 
\begin{alignat*}{2}
0& \leq \scal{y_{n+1}-x_\star}{2x_{n+1} -z_n - \gamma Cx_{n+1} -y_{n+1} -x_\star+z_\star+\gamma Cx_\star }\notag\\
&= \scal{y_{n+1}-x_\star}{ x_{n+1} -y_{n+1} -\gamma Cx_{n+1} + \gamma Cx_\star }+ \scal{y_{n+1}-x_\star}{u_{n+1}} \notag\\
&\leq \scal{y_{n+1}-x_\star}{ x_{n+1} -y_{n+1}+ \gamma Cy_{n+1}-\gamma Cx_{n+1}  }+ \scal{y_{n+1}-x_\star}{u_{n+1}},
\end{alignat*}
which implies that 
\begin{align}
&2\gamma\scal{y_{n+1}-x_\star}{Cx_{n+1} - Cy_{n+1}}\notag\\
&\quad\leq 2 \scal{y_{n+1}-x_\star}{x_{n+1}-y_{n+1}} + 2\scal{y_{n+1}-x_\star}{u_{n+1}}\notag\\
&\quad= \|x_{n+1}-x_\star\|^2 - \|y_{n+1}-x_\star\|^2 - \|x_{n+1}-y_{n+1}\|^2 + 2\scal{y_{n+1}-x_\star}{u_{n+1}}.
\label{e:eq2}
\end{align}
Combining \eqref{e:eq1} and \eqref{e:eq2}, we get
\begin{align*}
&\|y_{n+1} - x_\star +\gamma (Cx_{n+1}-Cy_{n+1})\|^2 \\
&\hspace{2cm}\leq \|x_{n+1}-x_\star\|^2 - \|x_{n+1}-y_{n+1}\|^2+  \gamma^2\|Cx_{n+1}-Cy_{n+1} \|^2\\
&\hspace{2cm}\quad+2\scal{y_{n+1}-x_\star}{u_{n+1}}.
\end{align*}
Applying this bound to \eqref{e:eq0}, we get
\begin{align}
\label{eq:3}
\|z_{n+1}-z_\star \|^2& \leq \|x_{n+1}-x_\star\|^2 - \|x_{n+1}-y_{n+1}\|^2 + \gamma^2\|Cx_{n+1}-Cy_{n+1} \|^2 \notag\\ &\quad-2\gamma\scal{Cx_{n+1}-Cy_{n+1}}{u_{n+1}} + \|u_{n+1}\|^2.
\end{align}

(i)
We consider the case where $B$ is $\kappa$-cocoercive.
From \eqref{eq:3}, we get
\begin{alignat}{2}
&\|z_{n+1}-z_\star \|^2\notag\\
&~\leq  \|z_n-z_\star\|^2- \|x_{n+1}-y_{n+1}\|^2+  \gamma^2\|Cx_{n+1}-Cy_{n+1} \|^2 \notag\\
&~\quad-2\gamma\scal{Cx_{n+1}-Cy_{n+1}}{u_{n+1}} -\frac{2\kappa}{\gamma}\|u_{n+1}\|^2\notag\\
 &~\leq
 \|z_n-z_\star \|^2- \|x_{n+1}-y_{n+1}\|^2 +\gamma^2 \left(1 +{\frac{\gamma}{2\kappa(1-\varepsilon')}}\right) \|Cx_{n+1}-C y_{n+1} \|^2 \notag\\
&~\quad -\frac{2\kappa\varepsilon'}{\gamma}\|u_{n+1}\|^2\notag
 \notag\\
&~\leq
 \|z_n-z_\star \|^2- \|x_{n+1}-y_{n+1}\|^2 +\gamma^2 \left(1 +{\frac{\gamma}{2\kappa(1-\varepsilon')}}\right)\mu^2 \|x_{n+1}- y_{n+1} \|^2   \notag\\
&~\quad-\frac{2\kappa\varepsilon'}{\gamma}\|u_{n+1}\|^2\notag\\
&~=  \|z_n-z_\star \|^2- \varepsilon \|x_{n+1}-y_{n+1}\|^2 -\frac{2\kappa\varepsilon'}{\gamma}\|u_{n+1}\|^2,\label{eq:case1_descent}
\end{alignat}
where $0<\varepsilon'<1$.
The first inequality follows from Lemma \ref{l:2},
the second inequality follows from Young's inequality,
the third inequality follows from $\mu$-Lipschitz continuity of $C$,
and the final equality follows from the definition
$\varepsilon:=1-\gamma^2 \left(1 +{\frac{\gamma}{2\kappa(1-\varepsilon')}}\right)\mu^2$.
We choose $\varepsilon'>0$ small enough so that $\varepsilon>0$.


(ii)
If $B= N_V$ and $C = P_V C_1 P_V$, then 
\begin{align*}
&\scal{Cx_{n+1}-Cy_{n+1}}{u_{n+1}} 
\\
 &\quad= \scal{C_1P_V x_{n+1} - C_1 P_V y_{n+1}}{P_V(x_{n+1}-z_n) + P_V(z_\star-x_\star)} = 0.
\end{align*}
Hence, \eqref{eq:3} becomes, 
\begin{align*}
&\|z_{n+1}-z_\star \|^2 \\
&\quad= \|x_{n+1}-x_\star\|^2 - \|x_{n+1}-y_{n+1}\|^2 +  \gamma^2\|Cx_{n+1}-Cy_{n+1} \|^2+ \|u_{n+1}\|^2 \notag\\
&\quad\leq \|z_n-z_\star \|^2 - \|x_{n+1}-y_{n+1}\|^2 +  \gamma^2\|Cx_{n+1}-Cy_{n+1} \|^2\notag\\
&\quad\leq \|z_n-z_\star \|^2 - \|x_{n+1}-y_{n+1}\|^2 +  \gamma^2\mu^2\|x_{n+1}-y_{n+1} \|^2\notag\\
&\quad=\|z_n-z_\star \|^2 - \varepsilon \|x_{n+1}-y_{n+1}\|^2,
\end{align*}
where the first inequality follows from $\|x_{n+1}-x_\star\|^2 \le  \|z_n-z_\star \|^2-\|u_{n+1}\|^2 $, which follows from $(1/2)$-averagedness of $P_V$,
 the second inequality follows from $\mu$-Lipschitz continuity of $C$,
 and the final equality follows from the definition $\varepsilon:=1-\gamma^2\mu^2 >0$. 

In cases (i) and (ii) both, we have 
\begin{equation*}
\label{e:ma1}
(\forall z_\star\in \Fix (T))\quad  \|z_{n+1}-z_\star \|^2 \leq  \|z_n-z_\star \|^2 - \varepsilon \|x_{n+1}-y_{n+1}\|^2
\end{equation*}
with $\varepsilon>0$,
which shows that $(z_n)_{n\in\NN}$ is Fej\'er monotone with respect to $\Fix(T)$ and  
\begin{equation*}
\sum_{n\in\NN} \|x_{n+1}-y_{n+1}\|^2 < +\infty,
\end{equation*}
which implies $x_{n+1}-y_{n+1}\rightarrow 0$.
Let us prove that every weak cluster point of $(z_n)_{n\in\NN}$ is in $\Fix(T)$. Let $\overline{z}$ be a weak cluster point of $(z_n)_{n\in\NN}$, i.e., there exists a subsequence $(z_{k_n})_{n\in\NN}$ 
such that $z_{k_n}\weakly \overline{z}$. Consider two cases: 

(i) We consider the case where $B$ is $\kappa$-cocoercive. From the second negative term in \eqref{eq:case1_descent}, we get
\begin{align*}
\sum_{n\in\NN} \|u_n\|^2 < +\infty& \quad\Longrightarrow\quad
u_n \to 0  
\quad\Longrightarrow \quad
x_{n+1} -z_n \to x_\star-z_\star \\& \quad\Longrightarrow\quad  Bx_{n+1}\to Bx_\star = B J_{\gamma B}z_\star
\end{align*}
where the last implication follows from $x_{n+1}=J_{\gamma B}z_n$, 
$z_{n} -x_{n+1}=\gamma Bx_{n+1}$,
$x_\star = J_{\gamma B}z_\star$,
and
$z_{\star}-x_{\star}=\gamma Bx_\star$.
Since $z_{k_n}\weakly \overline{z}$, we have $x_{1+k_n} \weakly \overline{x} = \overline{z} - \gamma Bx_\star$ and $y_{1+k_n} \weakly \overline{z} -\gamma Bx_\star$.
Since 
$x_{1+k_n}\weakly \overline{x}$ and $Bx_{1+k_n} \to Bx_\star$ and $\gra(B)$ is closed under $\HH^{\text{weak}} \times \HH^{\text{ strong}}$ \cite[Proposition 20.38]{livre1}, we get $Bx_\star = B\overline{x}$. Hence, $\overline{x} = J_{\gamma B}\overline{z}$.
By definition of the FDRF iteration, we have
\begin{equation*}
\underbrace{x_{1+k_n} -z_{k_n}}_{\to -\gamma B \overline{x}}
+ 
\underbrace{x_{1+k_n} -y_{1+k_n}}_{\to 0}  +
\underbrace{\gamma Cy_{1+k_n}-  \gamma Cx_{1+k_n}}_{\to 0}
\in
\gamma A\underbrace{y_{1+k_n}}_{\weakly \overline{x}}+ \gamma C \underbrace{y_{1+k_n}}_{\weakly \overline{x}}.
\end{equation*}
Since $A+C$ is maximal monotone ($A$ and $C$ are maximal monotone with $\dom C=\HH$  \cite[Corollary 25.5]{livre1})
$\gr(A+C)$ is closed under $\HH^{\text{weak}} \times \HH^{\text{ strong}}$ \cite[Proposition 20.38]{livre1} and we get 
\begin{equation*}
-\gamma B \overline{x} \in \gamma A \overline{x} + \gamma C \overline{x},
\end{equation*}
which shows that $\overline{x} \in \zer(A+B+C)$. Furthermore, $\overline{x}-\overline{z} \in \gamma A \overline{x} + \gamma C \overline{x}$ and hence 
$\overline{x} = J_{\gamma A}(2\overline{x} -\overline{z} -\gamma C\overline{x}) = J_{\gamma B}\overline{z}$. Therefore,
\[
(\Id -\gamma C)J_{\gamma A}(2\overline{x} -\overline{z} -\gamma C\overline{x}) + \overline{z} -  (\Id -\gamma C) J_{\gamma B}\overline{z} = \overline{z},
\]
or equivalently $T\overline{z} = \overline{z}$. Hence, $z_n\weakly z_{\star}$ and $x_n \weakly J_{\gamma B}z_{\star}$.
Since we have $x_{n+1}-y_{n+1}\rightarrow 0$, we conclude $y_n\weakly J_{\gamma B}z_{\star}$.

(ii) We consider the case where $B = N_V$. Then, $J_{\gamma B} = P_V$ is weakly continuous and hence $x_{1+k_n} \weakly \overline{x} = P_V \overline{z}$. Then, we have
\begin{alignat*}{2}
p_{k_n} &:= 
\underbrace{x_{1+k_n} -z_{k_n}}_{\weakly\overline{x} - \overline{z}} + 
\underbrace{x_{1+k_n} -y_{1+k_n}}_{\to 0} 
+\underbrace{\gamma Cy_{1+k_n} -  \gamma Cx_{1+k_n}}_{\to 0} \\
&\in\gamma A\underbrace{y_{1+k_n}}_{\weakly \overline{x}}+ \gamma C \underbrace{y_{1+k_n}}_{\weakly \overline{x}}.
\end{alignat*}
Since $A+C$ is maximal monotone, $x_{1+k_n}=P_Vz_{k_n}$, $x_{1+k_n}-y_{1+k_n}\rightarrow 0$, we have
\begin{equation*}
p_{k_n} \weakly \overline{x} - \overline{z}, \quad
P_{V^\perp} y_{1+k_n} \to 0,\quad \text{and}\quad P_{V} p_{k_n} \to 0.
\end{equation*}
From \cite[Example 26.7]{livre1} we have
$\overline{x} \in \zer(A+C+N_V)$ and $\overline{x} - \overline{z} \in (\gamma A+\gamma C)\overline{x}$.
Hence
$\overline{x} = J_{\gamma A}(2\overline{x} -\overline{z} -\gamma C\overline{x}) = J_{\gamma B}\overline{z}$. Therefore,
\[
(\Id -\gamma C)J_{\gamma A}(2\overline{x} -\overline{z} -\gamma C\overline{x}) + \overline{z} -  (\Id -\gamma C) J_{\gamma B}\overline{z} = \overline{z},
\]
or equivalently $T\overline{z} = \overline{z}$. Hence $z_n\weakly \overline{z}$ and $x_n \weakly J_{\gamma B}\overline{z}$.
\qed\end{proof}



Under condition (i), $B$ is single-valued and one can alternatively use FBF \cite{Tseng00} or FBHF \cite{Luis16} by utilizing forward evaluations of $B$ rather than the resolvent $J_{\gamma B}$.
However, many cocoercive operators $B$ require similar computational costs for evaluating $B$ and $J_B$,
and, in such cases, it may be advantageous to use $J_B$ instead of the forward evaluation $B$ \cite{plc17}.

Under condition (ii), the operator $B=N_V$ enforces a linear equality constraint.
Consider
\[
\mbox{find $x\in \HH$ \quad such that \quad }0\in \sum^m_{i=1}A_ix+C_i x,
\]
where $A_1,\dots,A_m$ are maximal monotone and $C_1,\dots,C_m$ are monotone and Lipschitz.
The equivalent formulation
\[
\mbox{find $\mathbf{x}\in \HH^m$ \quad such that \quad }
0\in N_V(\mathbf{x})+\sum^m_{i=1}(A_ix_i+C_ix_i),
\]
where $\mathbf{x}=(x_1,\dots,x_m)$ and $V=\menge{\mathbf{x}\in \HH^m}{x_1=\dots=x_m}$ is the consensus set,
is an important instance of case (ii).
(This problem class is the motivation for Raguet et al.'s forward-Douglas--Rachford \cite{Raguet11,Raguet18}.)
When $B=N_V$ and $V$ is the consensus set, FDRF reduces to Banert's relaxed forward backward, presented in the thesis \cite{Banert12}.
Finally, Brice\~no-Arias's forward--partial inverse--forward \cite{Luis15JOTA} is also applicable under this setup.
Brice\~no-Arias's method is different from our FDRF, but
it can also be considered a ``forward--Douglas--Rachford--forward splitting''
as it reduces to DRS and FBF as special cases.

FDRF resembles forward-Douglas--Rachford (FDR) splitting \cite{Raguet11,Raguet18,Luis15,Davis} but is different due to the correction term $\gamma(Cy_{n+1}-Cx_{n+1})$.
For convergence, FDR requires $C$ to be  cocoercive or (with a slight modification) $B$ to be  strongly monotone \cite[Theorems 1.1 and 1.2]{Davis}.
In contrast, Theorem~\ref{t:1} states that FDRF converges when $B$ is cocoercive.

However, FDRF does not converge in full generality.
The following result establishes that Assumptions \eqref{assump:a1}, \eqref{assump:a2}, and \eqref{assump:a3} are not sufficient to ensure that FDRF converges.
\begin{theorem} \label{t:divergence}
Given any $\gamma> 0$, there exist operators $A$, $B$, and $C$ satisfying Assumptions \eqref{assump:a1}, \eqref{assump:a2}, and \eqref{assump:a3}
such that the FDRF iterates $(z_n)_{n\in\NN}$ and $(x_n)_{n\in\NN}$ diverge.
\end{theorem}
\begin{proof}
Let $\HH=\RR^2$ and 
let $A$, $B$, and $C$ satisfy
\begin{gather*}
J_{\gamma A}(x,y)=
\begin{bmatrix}
0&0\\
0&0
\end{bmatrix}
\begin{bmatrix}
x\\y
\end{bmatrix},\qquad
B(x,y)=
\begin{bmatrix}
0&\gamma^{-1}\cot(\omega/2)\\
-\gamma^{-1}\cot(\omega/2)&0
\end{bmatrix}
\begin{bmatrix}
x\\y
\end{bmatrix},\\
C(x,y)=
\begin{bmatrix}
0&\mu\\
-\mu&0
\end{bmatrix}
\begin{bmatrix}
x\\y
\end{bmatrix}
,
\end{gather*}
where $\cot$ denotes cotangent and $\omega>0$ is small. 
Then, $A$, $B$, and $C$ are maximally monotone and $\{0\}=\zer(A+B+C)$.
With direct calculations, we get
\[
T(x,y)=
\begin{bmatrix}
\frac{1}{2}(1+\cos(\omega)+\gamma\mu \sin(\omega))&
\frac{1}{2}(\gamma\mu-\gamma\mu\cos(\omega)+ \sin(\omega))\\
\frac{1}{2}(-\gamma\mu+\gamma\mu\cos(\omega)- \sin(\omega))&
\frac{1}{2}(1+\cos(\omega)+\gamma\mu \sin(\omega))
\end{bmatrix}
\begin{bmatrix}
x\\y
\end{bmatrix}.
\]
The $2\times 2$ matrix defining $T$ has eigenvalues
\[
|\lambda_1|^2=|\lambda_2|^2=
(\cos(\omega/2)+\gamma\mu\sin(\omega/2))^2=1+\gamma\mu\omega+\mathcal{O}(\omega^2).
\]
So $|\lambda_1|^2=|\lambda_2|^2>1$ for small enough $\omega$.
Therefore, FDRF with $z_0\ne 0$ diverges in the sense that $\|z_n\|\rightarrow\infty$ and $\|x_n\|\rightarrow\infty$.
\qed\end{proof}

In splitting methods, step size requirements often depend on the assumptions, rather than on the specific operators.
Theorem~\ref{t:divergence} rules out the possibility of proving a result like
``Assuming (A1), (A2), and (A3), FDRF converges for $\gamma\in ]0,\gamma_\mathrm{max}(\mu)[$'', where $\gamma_\mathrm{max}(\mu)$ is some function that depends $\mu$.
However, Theorem~\ref{t:divergence} does not rule out the possibility that one can examine the specific operators $A$, $B$, and $C$ (to gain more information beyond the Lipschitz parameter of $C$) and then select $\gamma>0$ to obtain convergence.

\section{FRB+DR: Convergent in General}
\label{s:FRDR}
To solve Problem~\eqref{e:prob1}, we propose the following iteration
\makeatletter
\tagsleft@true
\makeatother
\begin{align}
\qquad\qquad\qquad
\begin{cases}
x_{n+1}=J_{\gamma B}(x_n-\gamma u_n-\gamma (2Cx_n-Cx_{n-1}))\\
y_{n+1}=J_{\beta A}(2x_{n+1}-x_n+\beta u_n)\\
u_{n+1}=u_n+\frac{1}{\beta}(2x_{n+1}-x_n-y_{n+1})
\end{cases}
\tag{FRDR}\label{eq:mts_drs}
\end{align}
\makeatletter
\tagsleft@false
\makeatother
for $n=0,1,\dots$, where $x_0,x_{-1},u_0\in\HH$ are starting points and $\gamma>0$, $\beta > 0$.
We call this method forward-reflected-Douglas--Rachford (FRDR) splitting
as it combines Malitsky and Tam's FRB \cite{Malitsky18} and Douglas--Rachford \cite{Lions1979}.
Note FRDR evaluates operator $C$ only once per iteration, since the evaluation of $Cx_{n-1}$ from the previous iteration can be reused.
In contrast, FDRF evaluates $C$ twice per iteration.

FRDR reduces to FRB when $A=0$ and to DR when $C=0$ and $\beta=\gamma$.
When $A=0$, we have $J_{\beta A}=\Id$, $u_n=0$, 
and the iteration is independent of $\beta$. FRB converges when $\gamma<1/(2\mu)$ \cite[Theorem 2.5]{Malitsky18}, which is consistent with the parameter range of Theorem~\ref{thm:MTS_DRS} with $\beta\rightarrow\infty$.
When $C=0$ and $\beta=\gamma$, one recovers DR with $z_n=x_n-\gamma u_n$.

Without any further assumptions, the $(x_n)_{n\in\NN}$ sequence of FRDR converges weakly to a solution of \eqref{e:prob1}.
\begin{theorem}
 \label{thm:MTS_DRS}
 Assume \eqref{assump:a1}, \eqref{assump:a2}, \eqref{assump:a3}, $0<\beta$,
 and $0<\gamma<\beta/(1+2\mu\beta)$.
 Then $x_n\weakly x_\star\in \zer(A+B+C)$ for \eqref{eq:mts_drs}.
\end{theorem}
\begin{proof}
Consider the Hilbert space $\HH\times\HH$ equipped with an alternative inner product and norm
\begin{align*}
\scal{(x,u)}{(y,v)}_{\HH\times\HH}&:=(1/\gamma)\scal{ x}{y}-\scal{x}{v}-\scal{ y}{u}+\beta \scal{u}{v},\\
\|(x,u)\|^2_{\HH\times\HH}&:=(1/\gamma)\|x\|^2-2\scal{x}{u}+\beta \|u\|^2.
\end{align*}
Since $\gamma<\beta$, the inner product and norm are valid.
Let $x_\star\in \zer(A+B+C)$, $u_\star\in Ax_\star$, and $-u_\star\in (B+C)x_\star$.

Define 
\begin{align*}
\tilde{A}y_{n+1}&:=u_n+\frac{1}{\beta}(2x_{n+1}-x_n-y_{n+1}),\\
\tilde{B}x_{n+1}&:=\frac{1}{\gamma}(x_n-x_{n+1})- u_n- 2Cx_n+Cx_{n-1},
\end{align*}
$\tilde{A}x_\star:=u_\star$, and $\tilde{B}x_\star:=-u_\star-Cx_\star$
so that
$\tilde{A}y_{n+1}\in Ay_{n+1}$,
$\tilde{B}x_{n+1}\in Bx_{n+1}$,
$\tilde{A}x_\star\in Ax_\star$, and
$\tilde{B}x_\star\in Bx_\star$.
Define
\begin{align*}
V_{n}&:=\|(x_n,u_n)-(x_\star,u_\star)\|^2_{\HH\times\HH}+\frac{1}{2}\|(x_n,u_n)-(x_{n-1},u_{n-1})\|^2_{\HH\times\HH}\\
&\qquad+2\scal{ Cx_{n}-Cx_{n-1}}{x_\star-x_{n}}\\
S_{n}&:=
\frac{1}{2}\|(x_n,u_n)-(x_{n+1},u_{n+1})\|^2_{\HH\times\HH}
+\frac{1}{2}\|(x_n,u_n)-(x_{n-1},u_{n-1})\|^2_{\HH\times\HH}.
\end{align*}

We have
\begin{align*}
&\|(x_{n+1},u_{n+1})-(x_{\star},u_{\star})\|^2_{\HH\times\HH}\\
&\qquad=
\|(x_{n},u_{n})-(x_{\star},u_{\star})\|^2_{\HH\times\HH}
-
\|(x_{n+1},u_{n+1})-(x_{n},u_{n})\|^2_{\HH\times\HH}\\
&\qquad\quad-
2\scal{\tilde{B}x_{n+1}-\tilde{B}x_\star}{x_{n+1}-x_\star}-
2\scal{ \tilde{A}y_{n+1}-\tilde{A}x_\star}{y_{n+1}-x_\star}\\
&\qquad\quad
+2\scal{ Cx_{n}-Cx_{n-1}}{x_{n}-x_{n+1}}
-2\scal{Cx_{\star}-Cx_{n}}{x_{\star}-x_{n+1}}\\
&\qquad\quad
+2\scal{Cx_{n}-Cx_{n-1}}{x_{\star}-x_{n}}\\
&\qquad
\stackrel{\text{(a)}}{\le}
\|(x_{n},u_{n})-(x_{\star},u_{\star})\|^2_{\HH\times\HH}
-
\|(x_{n+1},u_{n+1})-(x_{n},u_{n})\|^2_{\HH\times\HH}\\
&\qquad\quad
+2\scal{ Cx_{n}-Cx_{n-1}}{x_{n}-x_{n+1}}
-2\scal{Cx_{\star}-Cx_{n}}{x_{\star}-x_{n+1}}\\
&\qquad\quad
+2\scal{Cx_{n}-Cx_{n-1}}{x_{\star}-x_{n}}\\
&\qquad
\stackrel{\text{(b)}}{\le}
\|(x_{n},u_{n})-(x_{\star},u_{\star})\|^2_{\HH\times\HH}
-
\|(x_{n+1},u_{n+1})-(x_{n},u_{n})\|^2_{\HH\times\HH}\\
&\qquad\quad
+L\|x_{n}-x_{n-1}\|^2+L\|x_{n+1}-x_{n}\|^2
-2\scal{Cx_{n+1}-Cx_{n}}{x_{\star}-x_{n+1}}\\
&\qquad\quad
+2\scal{Cx_{n}-Cx_{n-1}}{x_{\star}-x_{n}}.
\end{align*}
Inequality (a) follows from monotonicity of $A$ and $B$.
Inequality (b) follows from 
\begin{align*}
2\scal{ Cx_{n}-Cx_{n-1}}{x_{n}-x_{n+1}}
&\le 
\frac{1}{\mu}
\|Cx_{n}-Cx_{n-1}\|^2+\mu\|x_{n}-x_{n+1}\|^2\\
&\le
\mu\left(\|x_{n}-x_{n+1}\|^2+\|x_{n}-x_{n-1}\|^2\right),
\end{align*}
which follows Young's inequality and Lipschitz continuity of $C$, and
\[
-2\scal{Cx_{\star}-Cx_{n}}{x_{\star}-x_{n+1}}
\le
-2\scal{Cx_{n+1}-Cx_{n}}{x_{\star}-x_{n+1}}
\]
which follows from monotonicity of $C$.
Reorganizing, we get
\begin{align*}
&V_{n+1}\\
&\le 
V_{n}-\frac{1}{2}
\left(
\|(x_{n+1},u_{n+1})-(x_{n},u_{n})\|^2_{\HH\times\HH}
+\|(x_{n},u_{n})-(x_{n-1},u_{n-1})\|^2_{\HH\times\HH}
\right)\\
&\quad+L\left(\|x_{n}-x_{n-1}\|^2+\|x_{n+1}-x_{n}\|^2\right)
\end{align*}
we now add
\begin{align*}
&\frac{\beta\gamma \mu}{\beta-\gamma}
\left(
\frac{1}{\beta}\|x_{n}-x_{n+1}\|^2-2\scal{ x_{n}-x_{n+1}}{u_{n}-u_{n+1}}+\beta\|u_{n}-u_{n+1}\|^2\right)\ge 0\\
&\frac{\beta\gamma \mu}{\beta-\gamma}
\left(\frac{1}{\beta}\|x_{n}-x_{n-1}\|^2-2\scal{x_{n}-x_{n-1}}{u_{n}-u_{n-1}}+\beta\|u_{n}-u_{n-1}\|^2
\right)\ge 0
\end{align*}
to the right-hand-side (nonnegativity follows from Young's inequality) to get
\begin{align}
V_{n+1}
&\le 
V_{n}-\frac{\beta-\gamma-2\mu\gamma\beta}{2(\beta-\gamma)}S_{n}.
\label{eq:thm2_key}
\end{align}

Using the telescoping sum argument with \eqref{eq:thm2_key}, we get
\begin{align*}
V_0-\frac{\beta-\gamma-2\mu\gamma\beta}{2(\beta-\gamma)}\sum^n_{i=0}S_i\ge V_n
\end{align*}
Next, we have
\begin{align*}
V_{n}
&
\stackrel{\text{(a)}}{\ge}
\frac{1}{\gamma}\|x_{n}-x_\star\|^2-2\scal{x_{n}-x_\star}{u_{n}-u_\star}+\beta\|u_{n}-u_\star\|^2\\
&\qquad+
\frac{1}{2\gamma}\|x_{n}-x_{n-1}\|^2-\scal{x_{n}-x_{n-1}}{u_{n}-u_{n-1}}+\frac{\beta}{2}\|u_{n}-u_{n-1}\|^2\\
&\qquad-
\mu\left( \| x_{n}-x_{n-1}\|^2+\|x_\star-x_{n}\|^2\right)\\
&\stackrel{\text{(b)}}{\ge}\frac{1}{2}\|(x_n,u_n)-(x_\star,u_\star)\|^2_{\HH\times\HH}\\
&\qquad+
\left(\frac{1}{2\gamma}-\frac{1}{2\beta}-\mu\right)\|x_{n}-x_{\star}\|^2
+
\left(\frac{1}{2\gamma}-\frac{1}{2\beta}-\mu\right)\|x_{n}-x_{n-1}\|^2\\
&\stackrel{\text{(c)}}{\ge} \frac{1}{2}\|(x_n,u_n)-(x_\star,u_\star)\|^2_{\HH\times\HH}.
\end{align*}
Inequality (a) follows from
\begin{align*}
2\scal{ Cx_{n}-Cx_{n-1}}{x_\star-x_{n}}
&\le 
\frac{1}{\mu}
\|Cx_{n}-Cx_{n-1}\|^2+\mu\|x_\star-x_{n}\|^2\\
&\le
\mu\left(\|x_{n}-x_{n-1}\|^2+\|x_\star-x_{n}\|^2\right),
\end{align*}
which follows Young's inequality and Lipschitz continuity of $C$,
inequality (b) follows from
\begin{align*}
\frac{\beta}{2}\|u_{n}-u_\star\|^2-\scal{x_{n}-x_\star}{u_{n}-u_\star}
&\ge 
-\frac{1}{2\beta}
\|x_{n}-x_\star\|^2\\
\frac{\beta}{2}
\|u_{n}-u_{n-1}\|^2-\scal{x_{n}-x_{n-1}}{u_{n}-u_{n-1}}
&\ge 
-\frac{1}{2\beta}
\|x_{n}-x_{n-1}\|^2,
\end{align*}
Young's inequality,
and inequality (c) follows from $\gamma<\beta/(1+2\mu\beta)$.
Putting these together, we have
\[
V_0-\frac{\beta-\gamma-2\mu\gamma\beta}{2(\beta-\gamma)}\sum^n_{i=0}S_i\ge 
\frac{1}{2}\|(x_n,u_n)-(x_\star,u_\star)\|^2_{\HH\times\HH}.
\]
This implies that the sequence $\left((x_n,u_n)\right)_{n\in \NN}$ is bounded and $S_n\rightarrow 0$.
Since
\begin{align*}
S_n&\ge
\frac{1}{2}\|(x_{n+1},u_{n+1})-(x_n,u_n)\|^2_{\HH\times\HH},
\end{align*}
$S_n\rightarrow 0$ implies $x_{n+1}-x_n\rightarrow 0$ and $u_{n+1}-u_n\rightarrow 0$.
Since
\[
u_{n+1}-u_n=(1/\beta)(2x_{n+1}-x_n-y_{n+1})
\]
we also have $x_{n+1}-y_{n+1}\rightarrow 0$.

Now consider a weakly convergent subsequence
$\left((x_{k_n},u_{k_n})\right)_{n\in \NN}$ such that $(x_{k_n},u_{k_n})\weakly (\overline{x},\overline{u})$.
Note that $x_{n+1}$ and $y_{n+1}$ are defined by the inclusion 
\begin{align*}
\begin{bmatrix}
\frac{1}{\gamma}\left(x_n-x_{n+1}\right)+2Cx_n-Cx_{n-1}-Cx_{n+1}\\
\frac{1}{\beta}\left(2x_{n+1}-y_{n+1}-x_{n}\right)
\end{bmatrix}
\in
\begin{bmatrix}
(B+C)x_{n+1}+u_n\\
Ay_{n+1}-u_n
\end{bmatrix}.
\end{align*}
The right-hand side is a maximal monotone operator on $\HH\times\HH$ (equipped with the usual inner product) \cite[Propositions 20.23, 20.38]{livre1}
and the left-hand side strongly converges to $0$ since $C$ is continuous.
Since $\gra(B+C)$ is closed under $\HH^{\text{weak}} \times \HH^{\text{ strong}}$ \cite[Proposition 20.38]{livre1}, we have
\begin{align*}
-\overline{u}&\in (B+C)\overline{x},\qquad
\overline{u}\in A\overline{x}.
\end{align*}
Adding these we also get
$\overline{x}\in \zer(A+B+C)$.
Finally, 
since $(V_n)_{n\in \NN}$ is a monotonically decreasing nonnegative sequence, it has a limit.
Since $C$ is continuous, $(x_n)_{n\in\NN}$ and  $(u_n)_{n\in\NN}$ are bounded sequences, and  $x_{n}-x_{n-1}\rightarrow 0$ and  $u_{n}-u_{n-1}\rightarrow 0$,
we have
\[
\lim_{n\rightarrow\infty}V_n=\lim_{n\rightarrow\infty}\|(x_n,u_n)-(x_\star,u_\star)\|^2_{\HH\times\HH}.
\]
By plugging in $(x_\star,u_\star)=(\overline{x},\overline{u})$,  we conclude that the entire sequence weakly converges to $(\overline{x},\overline{u})$.
\qed\end{proof}

The proof of Theorem~\ref{thm:MTS_DRS} closely follows Malitsky and Tam's analysis \cite[Lemma 2.4]{Malitsky18}.
In fact, FRDR can be thought of as an instance FRB on a primal-dual system with an auxiliary metric.
Naively translating Malitsky and Tam's convergence analysis via a change of coordinates leads to a step size requirement of
$\mu<(\beta-\gamma)/(1+\gamma\beta+\sqrt{(1+\gamma\beta)^2-4\gamma(\beta-\gamma)})$ and $0<\gamma<\beta$.
With a direct analysis, we obtain the better (and simpler) requirement of $\mu<(\beta-\gamma)/(2\beta\gamma)$.

The discovery of this proof was computer-assisted in the sense that we used the performance estimation problem (PEP) \cite{drori2014,Taylor2017,Ryu2018_pep} and a computer algebra system (CAS).
We briefly describe the strategy here.

The proof of Theorem~\ref{thm:MTS_DRS} crucially relies on finding the Lyapunov function $V_n$ and showing $V_{n+1}\le V_n-S_n$.
For a fixed numerical value of $\beta$, $\gamma$, and $\mu$, roughly speaking, the PEP allows us to pose a semidefinite program (SDP) which solution indicates whether a candidate Lyapunov function is nonincreasing.
(A proof establishing a Lyapunov function is nonincreasing is a nonnegative combination of known inequalities, and the SDP automates the search.)
We used the SDP to quickly experiment with many candidate Lyapunov functions and identified that the $V_n$ used in the proof of Theorem~\ref{thm:MTS_DRS} is a nonincreasing quantity.

To get a general proof, we numerically solved the SDP for many values of $\beta$, $\gamma$, and $\mu$ and deduced the general symbolic solution.
(The general proof is equivalent to a general solution of the SDP that symbolically depends on $\beta$, $\gamma$, and $\mu$.)
We relied on a CAS, Mathematica, to work through the symbolic calculations.
In deducing the symbolic form of the proof, we utilized the observed structure of the solution. For example, the optimal SDP matrices were rank deficient, so we set the determinant of the corresponding symbolic matrix to $0$ and eliminated a degree of freedom.

Finally, we translated the symbolic calculations into a traditional proof that is verifiable by humans without the aid of computer software.
This step involved some further simplifications, such as replacing the identity
\[
2\scal{u}{v}= \eta\|u\|^2+\frac{1}{\eta}\|v\|^2-\|\sqrt{\eta}u-(1/\sqrt{\eta})v\|^2
\]
with Young's inequality
\[
2\scal{u}{v}\le \eta\|u\|^2+\frac{1}{\eta}\|v\|^2.
\]

\section{Comparison with Other Methods}
\label{s:comparison}
We now quickly examine other existing methods applied to Problem \eqref{e:prob1} to see how they differ from FDRF and FRDR.
We leave the comparison of these methods, in terms of their computational effectiveness, as a direction of future work.
Note that Problem \eqref{e:prob1} can be reformulated into the primal dual system
\begin{equation}\label{e:Refor1}
\mbox{find $x,u\in \HH$ \quad such that \quad }
    \begin{bmatrix}
    0\\
    0
    \end{bmatrix}
    \in 
    \begin{bmatrix}
    Bx    \\A^{-1}u
    \end{bmatrix}
    +\begin{bmatrix}
    C & \Id\\
    -\Id & 0
    \end{bmatrix}
     \begin{bmatrix}
    x\\
    u
    \end{bmatrix}.
\end{equation}

\paragraph{Combettes--Pesquet.}
The method of \cite{plc6} can be thought of as FBF applied to the primal-dual system \eqref{e:Refor1}:
\begin{equation*}
    \begin{cases}
    \overline{x}_{n+1} = J_{\gamma B}(x_n -\gamma (Cx_n+ u_n))\\
    y_{n+1} = J_{\gamma^{-1}A}(x_n+\gamma^{-1}u_n)\\
    x_{n+1} = \overline{x}_{n+1} - \gamma (C\overline{x}_{n+1}- Cx_n)-\gamma^2(x_n-y_{n+1})\\
    u_{n+1} = u_n+\gamma (\overline{x}_{n+1}-y_{n+1}).
    \end{cases}
\end{equation*}
This method solves Problem \eqref{e:prob1} with an appropriate choice of $\gamma>0$.
This method does not reduce to DR when $C=0$.

\paragraph{Malitsky--Tam.}
We can plainly applying FRB \cite{Malitsky18} to \eqref{e:Refor1}:
\begin{equation*}
    \begin{cases}
    x_{n+1} = J_{\gamma B}(x_n -\gamma (2Cx_n - Cx_{n-1} + 2u_n -u_{n-1}))\\
    y_{n+1} = J_{\gamma^{-1} A}(2x_n - x_{n-1}+\gamma^{-1}u_n)\\
    u_{n+1} = u_n+\gamma(2x_n-x_{n-1}-y_{n+1}).
    \end{cases}
\end{equation*}
This method solves Problem \eqref{e:prob1} with an appropriate choice of $\gamma>0$.
This method does not reduce to DR when $C=0$.


\paragraph{Brice\~no-Arias.}
When $B=N_V$ and $V\subset\HH$ is a closed vector space, i.e., in case (ii) of Theorem~\ref{t:1},
 forward--partial inverse--forward \cite{Luis15JOTA} applies:
\begin{align*}
    \begin{cases}
x_{n+1}=J_{\gamma A}(z_n-\gamma J_{\gamma B}CJ_{\gamma B}z_n)\\
y_{n+1}=J_{\gamma B}(2x_{n+1}-z_n+\gamma J_{\gamma B}CJ_{\gamma B}z_n)-x_{n+1}+z_n-\gamma J_{\gamma B}CJ_{\gamma B}z_n\\
z_{n+1}=y_{n+1}-\gamma (J_{\gamma B}CJ_{\gamma B}y_{n+1}-J_{\gamma B}CJ_{\gamma B}z_{n+1}).
    \end{cases}
\end{align*}
This method reduces to DRS when $C=0$ and to FBF when $B=0$.
However, this method does not apply in the general setup when $B$ is not a normal cone operator.

\paragraph{Johnstone--Eckstein.}
The method of \cite{johnstone2018,johnstone2018b} is based on the notion of projective splitting and bears little resemblance the other methods.
The method is very flexible, and there are multiple instances applicable to Problem \eqref{e:prob1}.
The following instance follows the presentation of \cite{johnstone2018}:
\begin{align*}
    \begin{cases}
    x_{n+1}^A=J_{\gamma A}(z_n+\gamma w_n^A)\\
    x_{n+1}^B=J_{\gamma B}(z_n+\gamma w_n^B)\\
    x_{n+1}^C=z_n-\gamma(Cz_n- w_n^C)\\
    z_{n+1}=z_n-\frac{\alpha_n}{\gamma^2}\left(
    3z_n-x^A_{n+1}-x^B_{n+1}-\gamma Cx_{n+1}^C+\gamma(w_n^A+w_n^B+w_n^C)\right)\\
    w_{n+1}^A=w_{n}^A-\alpha_n(x_{n+1}^A-x_n^C)\\
    w_{n+1}^B=w_{n}^B-\alpha_n(x_{n+1}^B-x_n^C)\\
    w_{n+1}^C=-w_{n+1}^A-w_{n+1}^B.
    \end{cases}
\end{align*}
The scalar parameter $\alpha_n$ is computed each iteration with a formula given in \cite{johnstone2018}.
This method does not reduce to DR or FBF.

\section{Conclusions}
In this paper, we considered the monotone inclusion problem with a sum of 3 operators, in which 2 are monotone and 1 is monotone-Lipschitz,
and studied combinations of methods type ``forward-Douglas-Rachford-forward''.
We presented FDRF, a combination of DR and FBF, and showed that it converges with further assumptions, but not generally.
We then presented FRDR, a combination of DR and FRB, and showed that it converges in general.
Moreover, FRDR has a lower computational cost per iteration since it evaluates the monotone-Lipschitz operator only once per iteration.
Therefore, we conclude FRDR to be the better forward-Douglas-Rachford-forward method.

\begin{acknowledgements}
Ernest Ryu was partially supported by AFOSR MURI FA9550-18-1-0502, NSF Grant DMS-1720237, and ONR Grant N000141712162.
B$\grave{\text{\u{a}}}$ng C\^ong V\~u's research work was funded by Vietnam National Foundation for Science and Technology Development (NAFOSTED) under  Grant No. 102.01-2017.05.
\end{acknowledgements}

\end{document}